\documentclass[11pt]{article}
\usepackage{pifont}

\usepackage{color}

\usepackage{amssymb}
\usepackage{amsfonts}
\usepackage{graphicx}
\usepackage{amsmath}
\usepackage{mathrsfs}
\usepackage{psfrag}
\usepackage{latexsym}
\usepackage{setspace}
\usepackage{amsbsy}
\usepackage{epsfig}
\usepackage{enumerate}
\usepackage{bm}

\usepackage{indentfirst}

\usepackage{epstopdf}

\usepackage{fancyhdr}
\usepackage{longtable}
\usepackage{booktabs}
\usepackage{bm}
\usepackage{type1cm}
\usepackage{xcolor}
\usepackage{pifont}
\usepackage{titlesec}
\usepackage{float}
\usepackage{comment}
\usepackage{mathrsfs}
\usepackage{geometry}
\usepackage{ulem}

\usepackage{setspace}
\usepackage{amsthm}
\usepackage{extarrows}

\newtheorem{thm}{Theorem}[section]
\newtheorem{prop}{Proposition}[section]
\newtheorem{coro}{Corollary}[section]

\newtheorem{lem}{Lemma}[section]
\newtheorem{definition}{Definition}[section]
\newtheorem{remark}{Remark}[section]

\newcommand{\R}{\mathbb{R}}

\newcommand{\HE}{\mathcal{E}}

\newcommand{\F}{\mathcal{F}}

\newcommand{\bF}{\mathbb{F}}
\newcommand{\lt}{\left}
\newcommand{\rt}{\right}
\newcommand{\ep}{\varepsilon}
\newcommand{\ph}{\varphi}

\newcommand{\ga}{\Gamma^{t,x}}
\newcommand{\Y}{Y^{\ep}}
\newcommand{\YY}{\widetilde{Y}^{\ep}}
\newcommand{\Z}{Z^{\ep}}
\newcommand{\ZZ}{\widetilde{Z}^{\ep}}
\newcommand{\CW}{\Omega}

\numberwithin{equation}{section}
\title{Representation theorems for generators of BSDEs and the extended $g$-expectations in probability  spaces with general filtration}
\date{}

\author{Panyu Wu, Guodong Zhang$^*$\\
{\small Zhongtai Securities Institute for Financial Studies, Shandong University}\\ {\small Jinan 250100, China; Email: wupanyu@sdu.edu.cn; zhang\_gd@mail.sdu.edu.cn}\\
{\small$^*$Corresponding author}}

\begin{document}

\maketitle

\begin{abstract}
In this paper, we establish representation theorems for generators of backward stochastic differential equations (BSDEs in short) in probability spaces with general filtration from the  perspective of transposition solutions of BSDEs. As applications, we give a converse comparison theorem  for generators of BSDEs and also some characterizations to positive homogeneity, independence of $y$, subadditivity and convexity  of generators of BSDEs. Then, we extend concepts of $g$-expectations and conditional $g$-expectations to the probability spaces with general filtration and investigate their properties.

\par  $\textit{Keywords:}$ Filtration, Backward stochastic differential equations, Representation theorem, Converse comparison theorem, $g$-expectation.

\end{abstract}



\section{Introduction}
Let $T> 0$ and $(\CW,\F,\bF,P)$ be a  filtered probability space with  $\bF= \{\F_t\}_{t\in[0,T]}$ satisfied the usual condition, on which a $d$-dimensional standard Brownian motion $\{W_t\}_{t\in[0,T]}$ is defined. When $\bF$ is the  natural filtration (generated by the Brownian motion $\{W_t\}_{t\in[0,T]}$ and augmented by all the $P$-null sets), Pardoux and Peng (\cite{Peng1990}), established the well-posedness of a backward
stochastic differential equation (BSDE for short) of the type
\begin{equation*}
\left\{\begin{aligned}
&dY_t=-g(t,Y_t,Z_t)dt+Z_tdW_t,\quad t\in[0,T]\\
&Y_T=\xi\\
\end{aligned}\ ,\right.
\end{equation*}
provided the generator $g$ satisfies Lipschitz and square integrable conditions and the terminal $\xi$ is square integrable. With the solutions of BSDEs, Peng (\cite{Peng1997-2}) introduced a type of dynamically consistent nonlinear expectations---$g$-expectations  and conditional $g$-expectations. Similarly to the classical case, it is also proved that $g$-expectations and conditional $g$-expectations preserve all properties of classical expectations (except the linearity) in \cite{Peng1997-2}. Since then, the properties and applications of BSDEs or $g$-expectations  under the natural filtration probability spaces have been widely studied, see for example  \cite{Briand,Chen,Chenepstein,Coquethu,Huchen,Mayao,Peng1997,Gianin}.

Since BSDE theory is under the framework of natural filtration, it  has extensive applications  in the area of complete financial market, such as, the pricing of contingent claims, the theory of recursive utilities and the risk measurement. However, when the filtration is a non-natural filtration, the financial market is incomplete , the classical BSDE theory can not be applied. On the other hand,
it's well known that stochastic differential equation (SDE for short) theory is built on general complete filtered probability spaces, not only on the natural filtration space which is the BSDE theory  built on. So, whether in theory or in application, it is meaningful to study the properties of BSDEs in general filtration spaces.
However, in the general case  when the filtration $\mathbb{F}$ maybe larger than the natural filtration, the classical martingale representation theorem maybe fails
which plays a critical role for the well-posedness of BSDEs.
As far as we know, there are few works on BSDEs  under the general filtration probability space.  El Karoui, Huang (\cite{Karouihuang}) and Liang et al. (\cite{Lianglyons}) obtained the well-posedness of BSDEs with general filtration, without using the martingale representation theorem. However, in their framework, it is hard to analyze the properties of the solutions of BSDEs. Cohen and Elliott (\cite{CohenElliott})  gave the well-posedness and comparison theorem of the BSDEs with general filtration, under the assumption that $L^2(\Omega,\mathcal{F}_T, P)$ is separable. In this framework, Cohen (\cite{Cohen}) also investigated the $g$-expectations and conditional $g$-expectations on $L^2(\Omega,\mathcal{F}_T, \bF, P)$  for general filtration.
A general type of martingale representation theorem is
fundamental to the approach in \cite{CohenElliott,Cohen}. Royer considered the $g$-expectations and conditional $g$-expectations when the filtration is generated by both the Brownian motion and a Poisson random measure in \cite{Royer}.

In  \cite{LuZhang}, L\"{u} and Zhang  proved the well-posedness of some linear and semilinear BSDEs with general filtration, without using the martingale representation theorem. The point of their approach is to introduce a new notion of
solution, that is, the transposition solution, which coincides with the usual strong solution when the filtration is natural. A comparison theorem for transposition solutions is also presented.

During the numerous results on BSDEs in the natural filtration space, representation theorem for generator of BSDE is a very important result since it represents the generator $g$ of BSDE by the limit of solutions of the corresponding BSDEs. It is firstly established by Briand et al. (\cite{Briand}) for BSDEs whose generators satisfy Lipshcitz condition and two additional assumptions that $E[\sup_{0\le t\le T}|g(t,0,0)|^2]<\infty$ and $(g(t,y,z))_{t\in [0,T]}$ is continuous in $t$. Then it is generalized by a series of work of Jiang (\cite{Jiang2005a,Jiang2005,Jiang2006,Jiang2008}) to case that $g$ satisfies Lipschitz condition. For  non-Lipschitz condition, representation theorem for generators of BSDEs is further studied, see for example  \cite{Fanjiang,Fan2011,Liu,Mayao,Zhengli} and references therein.

In this paper, we will consider the properties of BSDEs on general filtration probability spaces  from the  perspective of transposition solutions of BSDEs.
The main results are the representation theorems for the generators of BSDEs which extend the representation theorems of \cite{Briand} to the  general filtration probability spaces.  As applications, we give a converse comparison theorem  for generators of BSDEs and also some characterizations to positive homogeneity, independence of $y$, subadditivity and convexity  of generators of BSDEs. Then we will extend the $g$-expectations and conditional $g$-expectations to $L^2(\Omega,\mathcal{F}_T, \bF, P)$  for general filtration $\bF$, and  we also establish
some properties of this generalized $g$-expectations.


The remainder of this paper is organized as follows. In Section 2, we introduce some preliminaries of transposition solutions of BSDEs. In Section 3, we establish the representation theorem for the generators and the converse comparison theorem of BSDEs  in general probability spaces. We also give some characterizations to generators of BSDEs.  In Section 4, we proceed to
generalize the definition of  $g$-expectations and conditional $g$-expectations in general probability spaces. Then we establish some properties of the extended  $g$-expectations, such as ``Zero-one" law, constant preserving and time consistency. We also obtain some necessary and sufficient conditions for positive homogeneity, translation invariance, subadditivity and convexity of this generalized g-expectations, respectively.

\section{Transposition Solution of BSDE}
Firstly, we introduce some notions which will be used in this paper. For $x, y\in\R^n$, $|x|$ denotes its Euclidian norm and $xy$ denotes the usual scalar product in $\mathbb{R}^n$. An $n\times d$ matrix will be considered as an element $z\in\R^{n\times d};$ note that its Euclidean norm is given by $|z|=\sqrt{trace(zz^*)}$.
For any $t\in[0,T]$, define\\
\small
$\bullet$ $L^2_{\F_t}(\CW,\R^n):=\{\xi\!:\CW\rightarrow\R^n|\ \xi \textit{ is } \F_t\textit{-measurable random variable and } \|\xi\|^2=E[|\xi|^2]<\infty\}$;\\
$\bullet$  $L^2_{\bF}(\CW,L^p(t,T;\R^{n\times d})):=\left\{\ph\!\!:[t,T]\times\CW\rightarrow\R^{n\times d}|\ \ph \textit{ is } \{\mathcal{F}_r\}_{r\in [t,T]} \textit{-adapted process and } \|\ph\|^2=E\lt[\lt(\int_t^T|\ph(r)|^pdr\rt)^{\frac{2}{p}}\rt]<\infty\right\}$, where $p\ge1$;\\
$\bullet$  $L^2_{\bF}(\CW,D([t,T];\R^n)):=\left\{\ph\!:[t,T]\times\CW\rightarrow\R^n|\ \ph \textit{ is } \{\mathcal{F}_r\}_{r\in [t,T]} \textit{-adapted } c\grave{a}dl\grave{a}g \textit{ process and }
  \|\ph\|^2=E\left[\sup\limits_{r\in[t,T]}|\ph(r)|^2\right]<\infty\right\}$;\\
$\bullet$  $L^2_{\bF}(\CW,C([t,T];\R^n)):=\left\{\ph\!:[t,T]\times\CW\rightarrow\R^n|\ \ph\in L^2_{\bF}(\CW,D([t,T];\R^n)\textit{ and }\ph \textit{ is continuous process}  \right\}.$
 \normalsize

Let us consider a function $g$, defined on $\CW\times[0,T]\times\R^n\times\R^{n\times d}$, with values in $\R^n$, such that the process $(g(t,y,z))_{t\in[0,T]}$ is progressively
measurable for each $(y,z)\in\R^n\times\R^{n\times d}$. From this point onwards, for notational simplicity, we shall regard $\omega$ as implicit in the function
$g$, whenever this does not lead to confusion.
In this paper, we may use  the following assumptions for  the function $g$:
\smallskip\\
(A1). There exists a constant $K\geq0$ such that $dP\times dt$-a.s.
$$|g(t,y,z)-g(t,y',z')|\leq K(|y-y'|+|z-z'|),\quad \quad \text{ for all } (y,z),(y',z')\in\R^n\times\R^{n\times d};$$
(A2). The process $(g(t,0,0))_{t\in[0,T]}\in L^2_{\bF}(\CW,L^1(0,T;\R^n));$\\
(A3). $E[\sup_{t\in[0,T]}|g(t,0,0)|^2]<\infty$;\color{black}\\
(A4). For all $(y,z)\in\R^n\times\R^{n\times d}$, $t\mapsto g(t,y,z)$ is right continuous in $[0,T[$ $P$-a.s.\\
(A5). For all $ y\in\R^n,\ g(t,y,0)=0$,  $dP\times dt$-a.s.
\smallskip

Under the assumptions (A1) and (A2), L\"{u} and Zhang (\cite{LuZhang}) proved that  the following BSDE has a unique  transposition solution,
\begin{equation}\label{bsde}
\left\{\begin{aligned}
&dY_t=-g(t,Y_t,Z_t)dt+Z_tdW_t,\quad t\in[0,T]\\
&Y_T=Y^T\\
\end{aligned}\ ,\right.
\end{equation}
where $Y^T\in L^2_{\F_T}(\CW,\R^n)$.

\begin{definition}[\cite{LuZhang} Definition 1.1]\label{de1}
We call $(Y_{\cdot},Z_{\cdot})\in L^2_{\bF}(\CW,D([0,T];\R^n))\times L^2_{\bF}(\CW,L^2(0,T;\R^{n\times d}))$ a transposition solution to (\ref{bsde}) if for any $t\in [0,T]$, $u_{\cdot}\in L^2_{\bF}(\CW,L^1(t,T;\R^n))$, $v_{\cdot}\in L^2_{\bF}(\CW,L^2(t,T;\R^{n\times d})) $ and $\eta\in L^2_{\F_t}(\CW,\R^n)$, the following identity
\begin{equation}\label{transp}
E\lt[Y^T X_T+\int_t^TX_rg(r,Y_r,Z_r)dr\rt]=E\lt[Y_t\eta+\int_t^Tu_rY_rdr+\int_t^Tv_rZ_rdr\rt]
\end{equation}
holds, where $X_{\cdot} \in L^2_{\bF}(\CW,C([t,T];\R^n))$ is the unique strong solution of the following SDE
\begin{equation}\label{sde}
\left\{\begin{aligned}
&dX_r=u_rdr+v_rdW_r,\quad r\in[t,T]\\
&X_t=\eta\\
\end{aligned}\ \right.
.\end{equation}
\end{definition}

\begin{prop}\label{prop3}
$(Y_{\cdot},Z_{\cdot})\in L^2_{\bF}(\CW,D([0,T];\R^n))\times L^2_{\bF}(\CW,L^2(0,T;\R^{n\times d}))$ is a transposition solution to (\ref{bsde}) if and only if for any $0\le s\le t\le T$, $u_{\cdot}\in L^2_{\bF}(\CW,L^1(s,T;\R^n))$, $v_{\cdot}\in L^2_{\bF}(\CW,L^2(s,T;\R^{n\times d})) $ and $\eta\in L^2_{\F_s}(\CW,\R^n)$, the following identity
\begin{equation*}
E\lt[Y_t X_t+\int_s^tX_rg(r,Y_r,Z_r)dr\rt]=E\lt[Y_s\eta+\int_s^tu_rY_rdr+\int_s^tv_rZ_rdr\rt]
\end{equation*}
holds, where $X_{\cdot}\in L^2_{\bF}(\CW,C([s,T];\R^n))$ is the unique strong solution of the following SDE
\begin{equation*}
\left\{\begin{aligned}
&dX_r=u_rdr+v_rdW_r,\quad r\in[s,t]\\
&X_s=\eta\\
\end{aligned}\ \right.
.\end{equation*}
\end{prop}

\begin{proof}[\textup{\bf Proof.}]
This result follows directly from Definition \ref{de1}, so we omit the details.
\end{proof}

\begin{thm}[\cite{LuZhang} Theorem 4.1]\label{th2}
For any given $Y^T\in L^2_{\F_T}(\CW,\R^n)$ and given function $g$ satisfied assumptions (A1) and (A2), the BSDE (\ref{bsde}) admits a unique transposition solution  $(Y_{\cdot},Z_{\cdot})\in L^2_{\bF}(\CW,D([0,T];\R^n))\times L^2_{\bF}(\CW,L^2(0,T;\R^{n\times d}))$.
Furthermore, there is a constant $C_{K,T} > 0$, depending only on $K$ and $T$, such that
\begin{equation}\label{estim}
\|(Y.,Z.)\|_{L^2_{\bF}(\CW,D([0,T];\R^n))\times L^2_{\bF}(\CW,L^2(0,T;\R^{n\times d}))}\leq C_{K,T}\lt[\|g(\cdot,0,0)\|_{L^2_{\bF}(\CW,L^1(0,T;\R^n))}+\|Y^T\|_{L^2_{\F_T}(\CW,\R^n)}\rt].
\end{equation}
\end{thm}

\begin{remark}
From the proof  of Theorem 4.1 in \cite{LuZhang}, we can see that $C_{K,T}$ can be uniformly bounded when $T$ is bounded. Thus, in the proof of Lemma \ref{rep-lem1} in this paper, we say that $C_{K,T}$ only depends on $K$, and denoted by $C_K$.
\end{remark}

\begin{remark}
In \cite{LuZhang}, the dimension of the Brownian motion is considered as $d=1$. All the definitions and results in \cite{LuZhang} can be extended to the case $d>1$ (see \cite{LuZhang2}).
\end{remark}


From now on, we consider $1$-dimensional BSDEs, that is $n=1$. The following theorem give a
comparison theorem by transposition solutions which is a little improved version of Theorem 5.1 in \cite{LuZhang}. The proof makes no essential difference, so we omit it.
\begin{thm}\label{comparison}
Let $(Y_{\cdot}^i,Z_{\cdot}^i)\in L^2_{\bF}(\CW,D([0,T];\R))\times L^2_{\bF}(\CW,L^2(0,T;\R^d)),i=1,2$ be the transposition solutions of the following BSDEs respectively,
\begin{equation*}
\left\{\begin{aligned}
&dY_t^i=-g_i(t,Y_t^i,Z_t^i)dt+Z_t^idW_t,\quad t\in[0,T],\\
&Y_T^i=\xi^i\\
\end{aligned}\right.\quad\quad i=1,2,
\end{equation*}
where $\xi^i\in L^2_{\F_T}(\CW,\R)$ and $g_i$ satisfy assumptions (A1) and (A2). If for $\xi^1\geq\xi^2$ $P$-a.s. and
$$dP \times dt\text{-a.s.},\quad g_1(t,Y^1_t,Z_t^1)\geq g_2(t,Y_t^1,Z_t^1),$$
then, for any $t\in[0,T]$
$$Y_t^1\geq Y_t^2,\quad P\text{-a.s.}$$
Moreover, this comparison is strict, that is, $Y_t^1=Y_t^2\ P$-a.s. for some $t\in[0,T]$ if and only if $\xi^1=\xi^2\ P$-a.s. and $g_1(r,Y^1_r,Z_r^1)= g_2(r,Y_r^1,Z_r^1)$ $P$-a.s. for a.e. $r\in[t,T]$.
\end{thm}
\section{Representation Theorem for Generators of BSDEs}
Based on the framework of transposition solution of BSDEs, we obtain the representation theorem for the generator of BSDE, which extends the representation  theorem (Proposition 2.3) of \cite{Briand} with weaker conditions to the general filtration spaces.
From now on, with out special notification, we always denote $\lt(Y_t\lt(g,T,\xi\rt),Z_t\lt(g,T,\xi\rt)\rt)_{t\in[0,T]}$  the unique transposition solution of BSDE with generator function $g$ and terminal condition $\xi$ at terminal time $T$.

Let $b(\cdot,\cdot,\cdot):\Omega\times[0,T]\times\R^m\rightarrow\R^m$ and $\sigma(\cdot,\cdot,\cdot):\Omega\times[0,T]\times\R^m\rightarrow\R^{m\times d}$ be two functions such that for any $x\in\R^m$, $b(\cdot,\cdot,x)$ and $\sigma(\cdot,\cdot,x)$  are both progressively measurable; and let $b$ and $\sigma$ satisfy the following assumptions (H1)-(H3), for notational simplicity, we also regard $\omega$ as implicit in $b$ and $\sigma$.
\smallskip\\
(H1). Lipschitz condition: There exists a  constant $L_1>0$ such that $P$-a.s.
$$|b(t,x)-b(t,x')|+|\sigma(t,x)-\sigma(t,x')|\leq L_1|x-x'|,\quad \text{ for all }x,x'\in\R^m,t\in[0,T];$$
(H2). Linear growth condition: There exists a  constant $L_2>0$ such that $P$-a.s.
$$|b(t,x)|+|\sigma(t,x)|\leq L_2(1+|x|),\quad \text{ for all }x\in\R^m,t\in[0,T];$$
(H3). For any $x\in\R^m$, $t\mapsto b(t,x)$  and $t\mapsto \sigma(t,x)$ are both right continuous in $[0,T[$ $P$-a.s.
\smallskip

For a fixed $(t,x,y,p)\in[0,T[\times\R^m\times\R\times\R^m$, let us denote by $\ga_.$  the solution of the following $m$-dimensional SDE
\begin{equation*}
\ga_s=x+\int_t^sb(r,\ga_r)dr+\int_t^s\sigma(r,\ga_r)dW_r,\quad t\leq s\leq T,
\end{equation*}
with the usual convention $\ga_s=x$ if $s<t$.

For any $\ep>0$ small enough, without loss of generality we always suppose $\ep<1$ in this section, the transposition solution of the BSDE
\begin{equation}\label{rep-bsde}
\left\{\begin{aligned}
&dY_s=-g(s,Y_s,Z_s)ds+Z_sdW_s,\quad s\in[0,t+\ep]\\
&Y_{t+\ep}=y+p\lt(\ga_{t+\ep}-x\rt)\\
\end{aligned}\ ,\right.
\end{equation}
is denoted by  $\lt(Y_s\lt(g,t+\ep,y+p(\ga_{t+\ep}-x)\rt),Z_s\lt(g,t+\ep,y+p(\ga_{t+\ep}-x)\rt)\rt)_{s\in[0,t+\ep]}$.
\begin{thm}\label{rep}
Let $b$, $\sigma$  be two functions satisfied assumptions (H1)-(H3).

(1). If $g$ satisfies assumptions (A1) and (A3), then, for any $(x,y,p)\in\R^m\times\R\times\R^m$,  we have
\begin{equation}\label{repthm}
g(t,y,\sigma^*(t,x) p)+pb(t,x)=\lim\limits_{\ep\to0^+}\frac{1}{\ep}\lt[Y_s\lt(g,t+\ep,y+p(\ga_{t+\ep}-x)\rt)-y\rt]\quad \text{in }L^2,
\end{equation}
holds for almost every $t\in[0,T[$.

(2). If $g$ satisfies assumptions (A1)(A3) and (A4), then (\ref{repthm}) holds for any $(t,x,y,p)\in[0,T[\times\R^m\times\R\times\R^m$.
\end{thm}

The proof of Theorem \ref{rep} requires the following lemmas, which was motivated by \cite{Jiang2005}.
\begin{lem}\label{rep-lem1}
Suppose $g,b,\sigma$ satisfy the assumptions in Theorem \ref{rep}, let $1\leq q\le 2$, then for any $(t,x,y,p)\in[0,T[\times\R^m\times\R\times\R^m$, the following two statements are equivalent:
\begin{description}
  \item{(1).} $\displaystyle g(t,y,\sigma^*(t,x) p)+pb(t,x)=\lim\limits_{\ep\to0^+}\frac{1}{\ep}\lt[Y_s\lt(g,t+\ep,y+p(\ga_{t+\ep}-x)\rt)-y\rt]\quad \text{in }L^q;$
  \item{(2).}  $\displaystyle g(t,y,\sigma^*(t,x) p)=\lim\limits_{\ep\to0^+}E\lt[\frac{1}{\ep}\int_t^{t+\ep}g(r,y,\sigma^*(t,x) p)dr\big|\mathcal{F}_t\rt] \quad\text{in }L^q.$
\end{description}

\end{lem}
\begin{proof}[\textup{\bf Proof}]
We only prove $``(2)\Rightarrow(1)"$,  $``(1)\Rightarrow(2)"$ is similar. For notational convenience, in this proof, we will write $\lt(\Y_s,\Z_s\rt)_{s\in[0,t+\ep]}$ instead of
{\small$\lt(Y_s\lt(g,t+\ep,y+p(\ga_{t+\ep}-x)\rt),Z_s\lt(g,t+\ep,y+p(\ga_{t+\ep}-x)\rt)\rt)_{s\in[0,t+\ep]}$    }.

By the classical results on SDEs, the terminal condition of the BSDE (\ref{rep-bsde}) is   square integrable and then $(\Y_s,\Z_s)_{s\in[0,t+\ep]}\in L^2_{\bF}(\CW,D([0,t+\ep];\R))\times L^2_{\bF}(\CW,L^2(0,t+\ep;\R^d))$ is the  unique  transposition solution of this BSDE, that is
\begin{equation}\label{repeq1}
E\lt[\Y_{t+\ep} X_{t+\ep}+\int_s^{t+\ep}X_rg(r,\Y_r,\Z_r)dr\rt]=E\lt[\Y_s\eta+\int_s^{t+\ep}u_r\Y_rdr+\int_s^{t+\ep}v_r\Z_rdr\rt], \text{ for all }  s\in[0,t+\ep],
\end{equation}
where  $X_{\cdot} \in L^2_{\bF}(\CW,C([s,t+\ep];\R))$ is the unique strong solution of the following SDE with any fixed $ u_{\cdot}\in L^2_{\bF}(\CW,L^1(s,t+\ep;\R)),v_{\cdot}\in L^2_{\bF}(\CW,L^2(s,t+\ep;\R^d)) $ and $\eta\in L^2_{\F_s}(\CW,\R)$,
\begin{equation*}
\left\{\begin{aligned}
&dX_r=u_rdr+v_rdW_r,\quad r\in[s,t+\ep]\\
&X_s=\eta\\
\end{aligned}\right..
\end{equation*}
For each $s\in[t,t+\ep]$,   define
$$\YY_s=\Y_s-\lt(y+p\lt(\ga_s-x\rt)\rt)\quad\text{ and } \quad\ZZ_s=\Z_s-\sigma^*(s,\ga_s)p.$$
Putting them into (\ref{repeq1}) and noticing that $\YY_{t+\epsilon}=0$, we have
\begin{align}\label{repeq2-0}
&E\lt[\lt(y+p\lt(\ga_{t+\ep}-x\rt)\rt)X_{t+\ep}+\int_s^{t+\ep}X_rg\lt(r,\YY_r+\lt(y+p(\ga_r-x)\rt),\ZZ_r+\sigma^*(r,\ga_r)p\rt)dr\rt]\nonumber\\
=&E\lt[\YY_s\eta+\lt(y+p\lt(\ga_s-x\rt)\rt)\eta+\int_s^{t+\ep}u_r\Y_rdr+\int_s^{t+\ep}v_r\Z_rdr\rt], \text{ for all }  s\in[t,t+\ep].
\end{align}
Using It\^{o}'s formula, we have
\begin{align*}
&\lt(y+p\lt(\ga_{t+\ep}-x\rt)\rt)X_{t+\ep}-\lt(y+p\lt(\ga_s-x\rt)\rt)\eta\\
=&\int_s^{t+\ep}X_rpb(r,\ga_r)dr+\int_s^{t+\ep}X_r\sigma^*(r,\ga_r)pdW_r+\int_s^{t+\ep}\lt(y+p\lt(\ga_r-x\rt)\rt)u_rdr\\
&+\int_s^{t+\ep}\lt(y+p\lt(\ga_r-x\rt)\rt)v_rdW_r+\int_s^{t+\ep}\sigma^*(r,\ga_r)pv_rdr\\
=&\int_s^{t+\ep}X_rpb(r,\ga_r)dr+\int_s^{t+\ep}\lt(\Y_r-\YY_r\rt)u_rdr+\int_s^{t+\ep}\lt(\Z_r-\ZZ_r\rt)v_rdr\\
&+\int_s^{t+\ep}X_r\sigma^*(r,\ga_r)pdW_r+\int_s^{t+\ep}\lt(y+p\lt(\ga_r-x\rt)\rt)v_rdW_r.
\end{align*}
Combining with (\ref{repeq2-0}), we have
\begin{align}\label{repeq2}
&E\lt[\int_s^{t+\ep}X_r\lt[g\lt(r,\YY_r+(y+p(\ga_r-x)),\ZZ_r+\sigma^*(r,\ga_r)p\rt)+pb(r,\ga_r)\rt]dr\rt]\nonumber\\
=&E\lt[\YY_s\eta+\int_s^{t+\ep}u_r\YY_rdr+\int_s^{t+\ep}v_r\ZZ_rdr\rt], \text{ for all }  s\in[t,t+\ep].
\end{align}
That is $(\YY_s,\ZZ_s)_{s\in[t,t+\ep]}\in L^2_{\bF}(\CW,D([t,t+\ep];\R))\times L^2_{\bF}(\CW,L^2(t,t+\ep;\R^d))$ is the  unique  transposition solution of the BSDE
\begin{equation*}
\left\{\begin{aligned}
&d\YY_s=-\lt[g\lt(s,\YY_s+(y+p(\ga_s-x)),\ZZ_s+\sigma^*(s,\ga_s)p\rt)+pb(s,\ga_s)\rt]ds+\ZZ_sdW_s,\quad s\in[t,t+\ep]\\
&\YY_{t+\ep}=0\\
\end{aligned}\ .\right.
\end{equation*}
Thus, by Theorem \ref{th2}, we have the following estimate,
\begin{align*}
E\lt[\sup\limits_{t\leq r\leq t+\ep}|\YY_r|^2+\int_{t}^{t+\ep}|\ZZ_r|^2dr\rt]\leq C_{K}E\lt[\lt(\int_t^{t+\ep}\lt|g\lt(r,y+p(\ga_r-x),\sigma^*(r,\ga_r)p\rt)+pb(r,\ga_r)\rt|dr\rt)^2\rt]
\end{align*}
where $C_{K}$ is a constant depend on $K$.  It can be deduced from (A1) and (H2) that
\begin{align}\label{estim2}
E\lt[\sup\limits_{t\leq r\leq t+\ep}|\YY_r|^2+\int_{t}^{t+\ep}|\ZZ_r|^2dr\rt]\leq C_{x,y,p}\ep^2E\lt[1+\sup\limits_{t\leq r\leq t+\ep}\lt(|\ga_r|^2+|g(r,0,0)|^2\rt)\rt]
\end{align}
where constant $C_{x,y,z}$ depends on $x,y,p,K,L_2$.

In (\ref{repeq2}), let $s=t$, and $u_r=0,v_r=0$, for all $r\in[t,t+\ep]$, we have  \begin{align*}
E\lt[\eta\int_t^{t+\ep}\lt[g\lt(r,\YY_r+(y+p(\ga_r-x)),\ZZ_r+\sigma^*(r,\ga_r)p\rt)+pb(r,\ga_r)\rt]dr\rt]=E\lt[\YY_t\eta\rt].
\end{align*}
Further more, we have
\begin{align}\label{repeq3}
&E\lt[\lt(\frac{1}{\ep}\YY_t-g(t,y,\sigma^*(t,x) p)-pb(t,x)\rt)\eta\rt]\nonumber\\
=&E\lt[\frac{\eta}{\ep}\int_t^{t+\ep}\lt[g\lt(r,\YY_r+(y+p(\ga_r-x)),\ZZ_r+\sigma^*(r,\ga_r)p\rt)-g(t,y,\sigma^*(t,x) p)+pb(r,\ga_r)-pb(t,x)\rt]dr\rt]\nonumber\\
=:&E\lt[R_t^{\ep}+Q_t^{\ep}+P_t^{\ep}+\eta \lt(E\lt[\frac{1}{\ep}\int_t^{t+\ep}g(r,y,\sigma^*(t,x) p)dr\big|\mathcal{F}_t\rt]-g(t,y,\sigma^*(t,x) p)\rt)\rt],
\end{align}
where $R_t^{\ep},Q_t^{\ep},P_t^{\ep}$ are denoted by
\begin{align*}
&R_t^{\ep}:=\frac{\eta}{\ep}\int_t^{t+\ep}\lt[g\lt(r,\YY_r+(y+p(\ga_r-x)),\ZZ_r+\sigma^*(r,\ga_r)p\rt)-g\lt(r, (y+p(\ga_r-x)), \sigma^*(r,\ga_r)p\rt)\rt]dr,\\
&Q_t^{\ep}:=\frac{\eta}{\ep}\int_t^{t+\ep}\lt[g\lt(r,(y+p(\ga_r-x)),\sigma^*(r,\ga_r)p\rt)-g(r,y,\sigma^*(r,x)p)+pb(r,\ga_r)-pb(r,x)\rt]dr,\\
&P_t^{\ep}:=\frac{\eta}{\ep}\int_t^{t+\ep}\lt[g(r,y,\sigma^*(r,x)p)-g(r,y,\sigma^*(t,x) p)+pb(r,x)-pb(t,x)\rt]dr.
\end{align*}

For a fixed $1< q\leq2$, now let $\eta=\lt(\frac{1}{\ep}\YY_t-g(t,y,\sigma^*(t,x) p)-pb(t,x)\rt)^{q-1}$, it will be checked that $\eta\in L^2_{\mathcal{F}_t}(\Omega,\mathbb{R})$. Due to assumption (A1),(H2) and estimate (\ref{estim2}), we have
\begin{align*}
\lt(E[|\eta|^2]\rt)^{\frac{1}{q-1}}\leq&E\lt[\lt|\frac{1}{\ep}\YY_t-g(t,y,\sigma^*(t,x) p)-pb(t,x)\rt|^2\rt]\\
\leq&\frac{3}{\ep^2}E\lt[\sup\limits_{t\leq r\leq t+\ep}|\YY_r|^2\rt]+6E\lt[\sup\limits_{0\leq r\leq T}|g(r,0,0)|^2\rt]+6K^2(|y|+|\sigma^*(t,x) p|)^2+3|pb(t,x)|^2\\
\leq &C_{x,y,p}'E\lt[1+\sup\limits_{0\leq r\leq T}|\ga_r|^2+\sup\limits_{0\leq r\leq T}|g(r,0,0)|^2\rt].
\end{align*}
where constant $C'_{x,y,z}$ depends on $x,y,p,K,L_2$.  Combining with the assumption (A3)  and the classical result of SDEs (\cite[Theorem 1.6.3]{YongZhou})
$$E\lt[\sup\limits_{0\leq r\leq T}|\ga_r|^2\rt]\leq C_T(1+|x|^2),\quad C_T>0\text{ depend on }T,$$
we have $E[|\eta|^2]<\infty$ and $\eta\in L^2_{\mathcal{F}_t}(\Omega,\mathbb{R})$.

It follows from $\eta=\lt(\frac{1}{\ep}\YY_t-g(t,y,\sigma^*(t,x) p)-pb(t,x)\rt)^{q-1}$ and (\ref{repeq3}) that
\begin{align*}
&E\lt[\lt|\frac{1}{\ep}(\Y_t-y)-g(t,y,\sigma^*(t,x) p)-pb(t,x)\rt|^q\rt]\\
=&E\lt[\lt|\lt(\frac{1}{\ep}\YY_t-g(t,y,\sigma^*(t,x) p)-pb(t,x)\rt)\eta\rt|\rt]\\
\leq&E \lt[|R_t^{\ep}|+|Q_t^{\ep}|+|P_t^{\ep}|\rt]+\lt(E\lt[|\eta|^{\frac{q}{q-1}}\rt]\rt)^{\frac{q-1}{q}} \lt(E\lt[\lt|E\lt[\frac{1}{\ep}\int_t^{t+\ep}g(r,y,\sigma^*(t,x) p)dr\big|\mathcal{F}_t\rt]-g(t,y,\sigma^*(t,x) p)\rt|^q\rt]\rt)^{\frac{1}{q}}.
\end{align*}
So, in order to deduce statement (1) from statement (2), we only need to prove $E[|R_t^{\ep}|+|Q_t^{\ep}|+|P_t^{\ep}|]\to0$ as $\ep\to0^+$.

Due to the assumption (A1) and the estimate (\ref{estim2}), with the H\"{o}lder's inequality, we have
\begin{align*}
\lt(E[|R_t^{\ep}|]\rt)^2\leq& \frac{K^2}{\ep^2}E\lt[|\eta|^2\rt]E\lt[\lt(\int_t^{t+\ep}(|\YY_r|+|\ZZ_r|)dr\rt)^2\rt]\\
\leq&\frac{2K^2}{\ep}E\lt[|\eta|^2\rt]E\lt[\ep\sup\limits_{t\leq r\leq t+\ep}|\YY_r|^2+\int_{t}^{t+\ep}|\ZZ_r|^2dr\rt]\\
\leq&\frac{2K^2}{\ep}E\lt[|\eta|^2\rt]E\lt[\sup\limits_{t\leq r\leq t+\ep}|\YY_r|^2+\int_{t}^{t+\ep}|\ZZ_r|^2dr\rt]\\
\leq&2K^2C_{x,y,p}\ep E\lt[|\eta|^2\rt]E\lt[1+\sup\limits_{t\leq r\leq t+\ep}\lt(|\ga_r|^2+|g(r,0,0)|^2\rt)\rt],
\end{align*}
With the fact that $E\lt[|\eta|^2\rt]<\infty$ and $E\lt[\sup\limits_{0\leq r\leq T}\lt(|\ga_r|^2+|g(r,0,0)|^2\rt)\rt]<\infty$,  we have $E[|R_t^{\ep}|]\to0$ as $\ep\to0^+$.

Similarly, it follows from assumptions (A1),(H1) and H\"{o}lder's inequality that
\begin{align*}
\lt(E[|Q_t^{\ep}|]\rt)^2\leq&\frac{2}{\ep^2}E\lt[|\eta|^2\rt]E\lt[ \lt(\int_t^{t+\ep}(K+KL_1+L_1)\lt|p(\ga_r-x)\rt|dr\rt)^2\rt]\\
\leq& C_{p,K,L_1}E\lt[|\eta|^2\rt]E\lt[\frac{1}{\ep}\int_t^{t+\ep}|\ga_r-x|^2dr\rt],\quad C_{p,K,L_1}>0\text{ depends on }p,K\text{ and }L_1.
\end{align*}
Since the function $r\mapsto E[|\ga_r-x|^2]$ is continuous (\cite[Theorem 1.6.3]{YongZhou} ), and this function is equal to $0$ at time $t$, we get $E[|Q_t^{\ep}|]\to0$ as $\ep\to0^+$.

By H\"{o}lder inequality, we also have
\begin{align*}
\lt(E[|P_t^{\ep}|]\rt)^2\leq& E\lt[|\eta|^2\rt]\frac{1}{\ep}E\lt[\int_t^{t+\ep}\lt|g(r,y,\sigma(r,x)p)-g(r,y,\sigma^*(t,x) p)+pb(r,x)-pb(t,x)\rt|^2dr\rt]\\
\leq&E\lt[|\eta|^2\rt]\frac{2|p|^2}{\ep}E\lt[\int_t^{t+\ep}K^2|\sigma(r,x)-\sigma(t,x)|^2+|b(r,x)-b(t,x)|^2dr\rt]
\end{align*}
With the assumption (H3), $\frac{1}{\ep}\int_t^{t+\ep}\lt(K^2|\sigma(r,x)-\sigma(t,x)|^2+|b(r,x)-b(t,x)|^2\rt)dr\to0$ as $\ep\to0^+$ $P$-a.s., further more, by assumption (H2) we have,
$$\frac{1}{\ep}\int_t^{t+\ep}\lt(K^2|\sigma(r,x)-\sigma(t,x)|^2+|b(r,x)-b(t,x)|^2\rt)dr\leq C_{K,L_2}\lt(1+|x|^2\rt)<\infty, $$
where $C_{K,L_2}>0$ depends on $K$ and $L_2$. Thus $E[|P_t^{\ep}|]\to0$ as $\ep\to0^+$, follows from Lebesgues dominated theorem.

 For the case $q=1$, we just let $\eta=1$, the other proof is similarly as the case $1<q\leq 2$.  The proof is completed.
\end{proof}
\begin{lem}[Lebeshue Lemma, see \cite{Hewitt} Lemma 18.4]\label{rep-lem2}
Let $f$ be a Lebesgue integrable function on the interval $[a,b]$. Then there exists a set $S\subseteq]a,b[$ such that $\mu([a,b]\setminus S)=0$, and for every $\alpha\in\R,t\in S$, we have
$$\lim\limits_{\ep\to0^+}\frac{1}{\ep}\int_{t}^{t+\ep}|f(s)-\alpha|=|f(t)-\alpha|,$$
where $\mu$ denotes the Lebesgue measure.
\end{lem}


\begin{proof}[{\bf Proof of Theorem \ref{rep}}] (1).
With the assumptions (A1) and (H2), for any $(t,x,y,p)\in[0,T]\times \R^m\times\R\times\R^m$ we have
$$
|g(t,y,\sigma^*(t,x) p)|\leq |g(t,0,0)|+K|y|+KL_2(1+|x|)|p|.
$$
Combining with the assumption (A3), we have $\{|\frac{1}{\ep}\int_t^{t+\ep}g(r,y,\sigma^*(t,x) p)dr|^2;\ep>0\}$ are uniformly integrable. It also can be checked that $|\int_0^Tg(r,y,\sigma^*(t,x) p)dr|<\infty, P$-a.s.. Follows from Lemma \ref{rep-lem2}, there exists a set $S\subseteq[0,T[$ such that $\mu([0,T]\setminus S)=0$, and for any $t\in S$ we have
$$\lim\limits_{\ep\to0^+}\frac{1}{\ep}\int_t^{t+\ep}g(r,y,\sigma^*(t,x) p)dr=g(t,y,\sigma^*(t,x) p),\ P\text{-a.s.}$$
Then we have, for any $t\in S$
$$g(t,y,\sigma^*(t,x) p)=\lim\limits_{\ep\to0^+}\frac{1}{\ep}\int_t^{t+\ep}g(r,y,\sigma^*(t,x) p)dr,\ \text{in }L^2$$
By Jensen's inequality we know that
\begin{align*}
&E\lt[\lt|E\lt[\frac{1}{\ep}\int_t^{t+\ep}g(r,y,\sigma^*(t,x) p)dr\big|\mathcal{F}_t\rt]-g(t,y,\sigma^*(t,x) p)\rt|^2\rt]\\
\leq&E\lt[\lt|\frac{1}{\ep}\int_t^{t+\ep}g(r,y,\sigma^*(t,x) p)dr-g(t,y,\sigma^*(t,x) p)\rt|^2\rt]
\end{align*}
Then, by Lemma \ref{rep-lem1}, we can finish  the proof of (1) in Theorem \ref{rep}.

(2). With the assumption (A4), for any $t\in[0,T[$, $\frac{1}{\ep}\int_t^{t+\ep}\lt|g(r,y,\sigma^*(t,x)p)-g(t,y,\sigma^*(t,x)p)\rt|^2dr\to0\ P$-a.s. as $\ep\to0^+$, further more,
$$\frac{1}{\ep}\int_t^{t+\ep}\lt|g(r,y,\sigma^*(t,x)p)-g(t,y,\sigma^*(t,x)p)\rt|^2dr\leq 2\lt(\sup\limits_{0\leq r\leq T}|g(r,0,0)|^2+K^2|y|^2+K^2L_2^2(1+|x|)^2|p|^2\rt)<\infty, $$
Follows from Lebesgues dominated theorem, we have
 \begin{align*}
&E\lt[\lt|E\lt[\frac{1}{\ep}\int_t^{t+\ep}g(r,y,\sigma^*(t,x) p)dr\big|\mathcal{F}_t\rt]-g(t,y,\sigma^*(t,x) p)\rt|^2\rt]\\
\leq&E\lt[\lt|\frac{1}{\ep}\int_t^{t+\ep}g(r,y,\sigma^*(t,x) p)dr-g(t,y,\sigma^*(t,x) p)\rt|^2\rt]\to0\quad \ep\to0^+.
\end{align*}
Then, by Lemma \ref{rep-lem1}, we can finish  the proof of (2) in Theorem \ref{rep}.
\end{proof}
\begin{coro}\label{rep-cor}
If $g$ satisfies assumptions (A1) and (A3), then, for any  $(y,z)\in\R\times\R^d$,    we have
\begin{equation}\label{coro}
g(t,y,z)=\lim\limits_{\ep\to0^+}\frac{1}{\ep}\lt[Y_t\lt(g,t+\ep,y+z(W_{t+\ep}-W_t)\rt)-y\rt]\quad \text{in }L^2.
\end{equation}
holds for almost every $t\in[0,T[$.

If further $g$ satisfies assumption (A4), then  (\ref{coro})
holds  for any $(t,y,z)\in[0,T[\times\R\times\R^d$.
\end{coro}

Since the representation theorem holds for almost every $t\in[0,T]$, from now on, for any $(y,z)\in\R\times\R^d$ we denote $S_{y,z}(g)$ the set of $t$ that makes the representation theorem of $g$ hold at $(y,z)$, that is,
$$S_{y,z}(g):=\lt\{t\in[0,T[\ \lt|\ g(t,y,z)=\lim\limits_{\ep\to0^+}\frac{1}{\ep}\lt[Y_t\lt(g,t+\ep,y+z(W_{t+\ep}-W_t)\rt)-y\rt]\quad  \text{in }L^2\rt.\rt\}$$

We are now in a position to present a converse comparison  theorem for transposition solutions of BSDEs which extends the Theorem 4.1 in \cite{Briand} with weaker condition to the general filtration case.

\begin{thm}\label{conversethm} Let $g_1,g_2$ be two functions satisfied assumptions (A1),(A3). If for all $T'\in[0,T]$, $\xi\in L^2_{\mathcal{F}_{T'}}$
 $$ Y_t(g_1,T',\xi)\leq Y_t(g_2,T',\xi), t\in[0,T']$$
then,  for all  $(y,z)\in\R\times\R^d$,
$$g_1(t,y,z)\leq g_2(t,y,z),  \quad dP\times dt\text{-a.s.}$$

If further, $g_1,g_2$ satisfy assumption (A4), then we have $P$-a.s.,  for all  $(t,y,z)\in[0,T]\times\R\times\R^d$, $g_1(t,y,z)\leq g_2(t,y,z).$
\end{thm}
\begin{proof}[{\bf Proof}]
For any fixed $(y,z)\in\times\R\times\R^d $ and $t\in S_{y,z}(g_1)\cap S_{y,z}(g_2)$, we have
\begin{equation*}
g_i(t,y,z)=\lim\limits_{n\to\infty}n\lt\{Y_t\lt(g_i,t+1/n,y+z(W_{t+1/n}-W_t)\rt)-y\rt\}\quad\text{ in }L^2, \quad i=1,2.
\end{equation*}
Then for any $t\in S_{y,z}(g_1)\cap S_{y,z}(g_2)$, there exists a subsequence $\{n_k\}_{k=1}^{\infty}$ of  $\{n\}_{n=1}^{\infty}$ such that,
\begin{equation*}
g_i(t,y,z)=\lim\limits_{k\to\infty}n_k\lt\{Y_t\lt(g_i,t+1/n_k,y+z(W_{t+1/n_k}-W_t)\rt)-y\rt\}\quad P\text{-a.s.}, \quad i=1,2.
\end{equation*}
On the other hand, by the hypothesis we deduce that
$$n_k\lt\{Y_t\lt(g_1,t+1/n_k,y+z(W_{t+1/n_k}-W_t)\rt)-y\rt\}\leq n_k\lt\{Y_t\lt(g_2,t+1/n,y+z(W_{t+1/n_k}-W_t)\rt)-y\rt\}\quad P\text{-a.s.}$$
Then for any $t\in S_{y,z}(g_1)\cap S_{y,z}(g_2)$, we obtain, the inequality $g_1(t,y,z)\leq g_2(t,y,z),  \ P\text{-a.s.}$

By Corollary \ref{rep-cor} we know that $\mu\lt([0,T]\setminus(S_{y,z}(g_1)\cap S_{y,z}(g_2))\rt)=0$, where $\mu$ denotes the Lebesgue measure. Then we have, for all  $(y,z)\in\R\times\R^d$,
$g_1(t,y,z)\leq g_2(t,y,z),\  dP\times dt\text{-a.s.}$
\end{proof}

\begin{thm}\label{th1}
Let $g$ be any given function satisfied assumptions (A1) and (A2).
\begin{enumerate}[(1).]
  \item\textup{\bf (Positive homogeneity)} If $g$ is positive homogeneous in $(y,z)$, that is, for all $(y,z)\in\R\times\R^d$ and $\alpha\geq0$,
      $$g(t,\alpha y,\alpha z)= \alpha g(t,y,z)\quad dP\times dt\text{-a.s.}$$
      Then for all $T'\in]0,T]$ any  $\xi\in L^2_{\F_{T'}}(\CW,\R)$,
      $$\lt(Y_{t}\lt(g,T',\alpha \xi\rt),Z_{t}\lt(g,T',\alpha \xi\rt)\rt)_{t\in[0,T']}=(\alpha Y_{t}\lt(g,T',\xi\rt),\alpha Z_{t}\lt(g,T',  \xi\rt))_{t\in[0,T']} $$
       in $ L^2_{\bF}(\CW,D([0,T'];\R))\times L^2_{\bF}(\CW,L^2(0,T';\R^{d})).$
  \item\textup{\bf (Translation invariance)} If $g$ is independent of $y$, then for any $T'\in[0, T], t\in[0,T']$ and $\xi\in L^2_{\F_{T'}}(\CW,\R),\beta\in L^2_{\F_t}(\CW,\R)$,
        $$\lt(Y_{s}\lt(g,T', \xi+\beta\rt),Z_{s}\lt(g,T', \xi+\beta\rt)\rt)_{s\in[t,T']}=( Y_{s}\lt(g,T',\xi\rt)+\beta, Z_{s}\lt(g,T',  \xi\rt))_{s\in[t,T']} $$
 in $ L^2_{\bF}(\CW,D([t,T'];\R))\times L^2_{\bF}(\CW,L^2(t,T';\R^{d})).$
  \item\textup{\bf (Sub-additivity)} If $g$ is sub-additive in $(y,z)$, that is,  for all $(y_1,z_1),(y_2,z_2)\in\R\times\R^d$,
      $$g(t,y_1+y_2,z_1+z_2)\leq g(t,y_1,z_1)+g(t,y_2,z_2)\quad dP\times dt\text{-a.s.}$$
      Then for any $T'\in[0, T]$ and   $ \xi_1,\xi_2\in L^2_{\F_{T'}}(\CW,\R)$,
      $$Y_{t}\lt(g,T', \xi_1+\xi_2\rt)\le Y_{t}\lt(g,T', \xi_1\rt)+Y_{t}\lt(g,T', \xi_2\rt)\quad \forall t\in[0, T'] \quad P\text{-a.s.}$$
  \item\textup{\bf (Convexity)} If $g$ is convex in $(y,z)$, that is, for all $(y_1,z_1),(y_2,z_2)\in\R\times\R^d$ and $\alpha\in[0,1]$,
      $$g(t,\alpha y_1+(1-\alpha)y_2,\alpha z_1+(1-\alpha)z_2)\leq \alpha g(t,y_1,z_1)+(1-\alpha)g(t,y_2,z_2)\quad dP\times dt\text{-a.s.}$$
     Then for any $T'\in[0, T]$ and   $\xi_1,\xi_2\in L^2_{\F_{T'}}(\CW,\R)$,
   $$Y_{t}\lt(g,T', \alpha\xi_1+(1-\alpha)\xi_2\rt)\le \alpha Y_{t}\lt(g,T', \xi_1\rt)+(1-\alpha)Y_{t}\lt(g,T', \xi_2\rt)\quad \forall t\in[0, T'] \quad P\text{-a.s.}$$
\end{enumerate}
\end{thm}

\begin{proof}[\textup{\bf Proof}]
We only give the proof when $T'=T$, the other situation is similar. For notational convenience, we will write $(Y_{t}^{g,\xi},Z_{t}^{g,\xi})_{t\in[0,T]}$ instead of \small$\lt(Y_{t}\lt(g,T, \xi\rt),Z_{t}\lt(g,T, \xi\rt)\rt)_{t\in[0,T]}$ \normalsize to denote the transposition solutions of BSDE with generator $g$ and terminal condition $\xi$ in this proof.

(1). For any $\alpha\ge 0$, $t\in[0,T]$, due to the positive homogeneity of $g$ and identity (\ref{transp}), we have
\begin{align*}
E\lt[\alpha\xi X_T+\int_t^T X_rg(r,\alpha Y^{g,\xi}_r,\alpha Z^{g,\xi}_r)dr\rt]=&E\lt[\alpha\xi X_T+\int_t^T\alpha X_rg(r,Y^{g,\xi}_r,Z^{g,\xi}_r)dr\rt]\\
=&E\lt[(\alpha Y_t^{g,\xi})\eta+\int_t^T u_r(\alpha Y^{g,\xi}_r)dr+\int_t^T v_r(\alpha Z^{g,\xi}_r)dr\rt].
\end{align*}
By the definition of transposition solution (Definition \ref{de1}), we have
$$\lt(Y_{t}^{g,\alpha \xi},Z_{t}^{g,\alpha \xi}\rt)_{t\in[0,T]}=\lt(\alpha Y_{t}^{g,\xi},\alpha Z_{t}^{g,  \xi}\rt)_{t\in[0,T]} $$
       in $ L^2_{\bF}(\CW,D([0,T];\R))\times L^2_{\bF}(\CW,L^2(0,T;\R^{d})).$

(2). For any $s\in[t,T]$, because of  (\ref{transp}) and $g$ being independent of $y$,  we have
\begin{align*}
E\lt[(\xi+\beta) X_T+\int_s^T X_rg(r,Y_r^{g,\xi}+\beta, Z_r^{g,\xi})dr\rt]=&E\lt[\beta(\eta+\int_s^Tu_rdr)+\xi X_T+\int_s^T X_rg(r,Y_r^{g,\xi},Z_r^{g,\xi})dr\rt]\\
=&E\lt[( Y_s^{g,\xi}+\beta)\eta+\int_s^T u_r( Y_r^{g,\xi}+\beta)ds+\int_s^T v_rZ_r^{g,\xi}dr\rt].
\end{align*}
By the definition of transposition solution (Definition \ref{de1}), we have
$$\lt(Y_{s}^{g, \xi+\beta},Z_{s}^{g, \xi+\beta}\rt)_{s\in[t,T]}=\lt( Y_{s}^{g,\xi}+\beta, Z_{s}^{g,  \xi}\rt)_{s\in[t,T]} $$
 in $ L^2_{\bF}(\CW,D([t,T];\R))\times L^2_{\bF}(\CW,L^2(t,T;\R^{d})).$

(3). For any $t\in[0,T]$, owing to (\ref{transp}), we have
\begin{align*}
&E\lt[(\xi_1+\xi_2)X_T+\int_t^TX_r\lt(g(r,Y_r^{g,\xi_1},Z_r^{g,\xi_1})+g(r,Y_r^{g,\xi_2},Z_r^{g,\xi_2})\rt)dr\rt]\\
=&E\lt[(Y_t^{g,\xi_1}+Y_t^{g,\xi_2})\eta+\int_t^Tu_r(Y_r^{g,\xi_1}+Y_r^{g,\xi_2})dr+\int_t^Tv_r(Z_r^{g,\xi_1}+Z_r^{g,\xi_2})dr\rt].
\end{align*}
For any $(\omega,t,y,z)\in\CW\times[0,T]\times\R\times\R^d $, define $$\bar{g}(\omega,t,y,z)=g\lt(\omega,t,y-Y_t^{g,\xi_2}(\omega),z-Z_t^{g,\xi_2}(\omega)\rt)+g\lt(\omega,t,Y_t^{g,\xi_2}(\omega),Z_t^{g,\xi_2}(\omega)\rt).$$
For any  $\xi_1,\xi_2\in L^2_{\F_T}(\CW,\R)$, it can be easily checked that
$$\lt(Y_{t}^{\bar{g}, \xi_1+\xi_2},Z_{t}^{\bar{g}, \xi_1+\xi_2}\rt)_{t\in[0,T]}=\lt( Y_{t}^{g,\xi_1}+Y_{t}^{g,\xi_2}, Z_{t}^{g,  \xi_1}+Z_{t}^{g,  \xi_2}\rt)_{t\in[0,T]} $$
 in $ L^2_{\bF}(\CW,D([0,T];\R))\times L^2_{\bF}(\CW,L^2(0,T;\R^{d})).$
Since $g$ is sub-additive in $(y,z)$, we have, $g(t,y,z)\leq \bar{g}(t,y,z)$ $dP\times dt$-a.s. It follows from the   comparison theorem (Theorem \ref{comparison}) that for all $t\in [0,T]$,
$$Y_{t}^{g, \xi_1+\xi_2}\leq Y_{t}^{\bar{g}, \xi_1+\xi_2}=Y_{t}^{g,\xi_1}+Y_{t}^{g,\xi_2}\ P\text{-a.s.}$$

(4). For any $\alpha\in[0,1]$, $t\in[0,T]$, thanks to (\ref{transp}), we have
\begin{align*}
&E\lt[(\alpha\xi_1+(1-\alpha)\xi_2)X_T+\int_t^TX_r\lt(\alpha g(r,Y_r^{g,\xi_1},Z_r^{g,\xi_1})+(1-\alpha)g(r,Y_r^{g,\xi_2},Z_r^{g,\xi_2})\rt)dr\rt]\\
=&E\lt[(\alpha Y_t^{g,\xi_1}+(1-\alpha)Y_t^{g,\xi_2})\eta+\int_t^Tu_r(\alpha Y_r^{g,\xi_1}+(1-\alpha)Y_r^{g,\xi_2})dr+\int_t^Tv_r(\alpha Z_r^{g,\xi_1}+(1-\alpha)Z_r^{g,\xi_2})dr\rt].
\end{align*}
For any $(\omega,t,y,z)\in\CW\times[0,T]\times\R\times\R^d $, define
\small$$\tilde{g}(\omega,t,y,z)=\alpha g\lt(\omega,t,\frac{1}{\alpha}(y-(1-\alpha)Y_t^{g,\xi_2}(\omega)),\frac{1}{\alpha}(z-(1-\alpha)Z_t^{g,\xi_2}(\omega))\rt)+(1-\alpha)g\lt(\omega,t,Y_t^{g,\xi_2}(\omega),Z_t^{g,\xi_2}(\omega)\rt).$$
\normalsize
For any  $\xi_1,\xi_2\in L^2_{\F_T}(\CW,\R)$, it can be easily checked that
\begin{align*}
\lt(Y_{t}^{\tilde{g}, \alpha\xi_1+(1-\alpha)\xi_2},Z_{t}^{\tilde{g}, \alpha\xi_1+(1-\alpha)\xi_2}\rt)_{t\in[0,T]}
=\lt( \alpha Y_{t}^{g,\xi_1}+(1-\alpha)Y_{t}^{g,\xi_2}, \alpha Z_{t}^{g,  \xi_1}+(1-\alpha)Z_{t}^{g,  \xi_2}\rt)_{t\in[0,T]}
\end{align*}
in  $L^2_{\bF}(\CW,D([0,T];\R))\times L^2_{\bF}(\CW,L^2(0,T;\R^{d}))$.
It follows from the convexity of  $g$ in $(y,z)$ that
\begin{align*}
g(t,y,z)&\leq\alpha g\lt(t,\frac{1}{\alpha}(y-(1-\alpha)Y_t^{g,\xi_2}),\frac{1}{\alpha}(z-(1-\alpha)Z_t^{g,\xi_2})\rt)+(1-\alpha)g\lt(t,Y_t^{g,\xi_2},Z_t^{g,\xi_2}\rt)\\
&=\bar{g}(t,y,z)\quad dP\times dt\text{-a.s.}
\end{align*}
Then, by the  comparison theorem (Theorem \ref{comparison}), we have for all $t\in [0,T]$,
$$Y_{t}^{g, \alpha\xi_1+(1-\alpha)\xi_2}\le Y_{t}^{\tilde{g}, \alpha\xi_1+(1-\alpha)\xi_2}= \alpha Y_{t}^{g, \xi_1}+(1-\alpha)Y_{t}^{g, \xi_2}\quad P\text{-a.s.}$$
\end{proof}
\begin{remark}
The  positive homogeneity and subadditivity of $g$  together are also known as sublinearity, which implies the convexity of $g$. Furthermore, from a remark of Briand et al. \cite{Briand}, the  assumption (A1) and (A5) and the convexity of $g$ imply that $g$ does not depend on y.
\end{remark}

If the assumption (A2) in Theorem \ref{th1} is strengthened to (A3), then the necessary conditions can also be sufficient.
\begin{thm}\label{converprop}
Let $g$ be any given function satisfied assumptions (A1) and (A3).
\begin{enumerate}[(1).]
  \item\textup{\bf (Positive homogeneity)} If for any $T'\in[0,T]$, $\xi\in L^2_{\F_{T'}}(\CW,\R)$,  and $\alpha\geq0$,
      \begin{equation}\label{posieq}
      Y_{t}\lt(g,T',\alpha \xi\rt)=\alpha Y_{t}\lt(g,T',\xi\rt),\quad \forall t\in[0,T']\quad P\text{-a.s.}
      \end{equation}
  Then $g$ is positive homogeneous in $(y,z)$.
  \item\textup{\bf (Translation invariance)} If  for any  $T'\in[0,T],\xi\in L^2_{\F_{T'}}(\CW,\R),$  and $c\in\R$,
    \begin{equation}\label{traneq}
    Y_{t}\lt(g,T', \xi+c\rt)= Y_{t}\lt(g,T',\xi\rt)+c\quad   \forall t\in[0,T']\quad P\text{-a.s.}
    \end{equation}
     Then  $g$ is independent of $y$.
  \item\textup{\bf (Sub-additivity)} If for any  $T'\in[0,T],\xi_1,\xi_2\in L^2_{\F_{T'}}(\CW,\R)$,
    \begin{equation}\label{subeq}
    Y_{t}\lt(g,T', \xi_1+\xi_2\rt)\leq Y_{t}\lt(g,T',\xi_1\rt)+Y_{t}\lt(g,T',\xi_2\rt)\quad   \forall t\in[0,T']\quad P\text{-a.s.}
    \end{equation}
     Then $g$ is  sub-additive in $(y,z)$.
  \item\textup{\bf (Convexity)} If  for any  $T'\in[0,T],\xi_1,\xi_2\in L^2_{\F_{T'}}(\CW,\R)$, and $\alpha\in[0,1]$,
 \begin{equation}\label{coneq}
 Y_{t}\lt(g,T', \alpha\xi_1+(1-\alpha)\xi_2\rt)\le \alpha Y_{t}\lt(g,T', \xi_1\rt)+(1-\alpha)Y_{t}\lt(g,T', \xi_2\rt)\   \forall t\in[0,T']\quad P\text{-a.s.}
 \end{equation}
    Then $g$ is convex in $(y,z)$.
\end{enumerate}
\end{thm}

\begin{proof}[\textup{\bf Proof}]
(1). For the case $\alpha=0$, it is trivial.
For the case $\alpha>0$, for any $T'\in[0,T]$ and $(\omega,t,y,z)\in \Omega\times[0,T']\times\R\times\R^d$, define
$$g^{\alpha}(\omega,t,y,z):=\frac{1}{\alpha}g\lt(\omega,t,\alpha y,\alpha z\rt).$$
It is clear that $g^{\alpha}$ also satisfies assumptions (A1) and (A3). By the definition and the uniqueness of transposition solution (Definition \ref{de1}, Theorem \ref{th1}), we have that, for any $\xi\in L^2_{\F_{T'}}(\CW,\R)$,
$$\lt(Y_{t}\lt(g^\alpha,T',\frac{\xi}{\alpha}\rt),Z_{t}\lt(g^{\alpha},T',\frac{\xi}{\alpha}\rt)\rt)_{t\in[0,T']}=\lt( \frac{1}{\alpha}Y_{t}\lt(g,T',\xi\rt),\frac{1}{\alpha}Z_{t}\lt(g,T',  \xi\rt)\rt)_{t\in[0,T']} $$
in $ L^2_{\bF}(\CW,D([0,T'];\R))\times L^2_{\bF}(\CW,L^2(0,T';\R^{d})).$

Combining this equality with (\ref{posieq}), we obtain for any $t\in[0,T']$, $\xi\in L^2_{\F_{T'}}(\CW,\R)$,
$$Y_{t}\lt(g^\alpha,T', \xi\rt)= Y_{t}\lt(g,T',\xi\rt)\quad P\text{-a.s.}$$
Then by the  converse comparison theorem Theorem \ref{conversethm}, we have for any $\alpha>0$,
$$\text{for all } (y,z)\in\R\times\R^d ,\ g(t,y,z)=g^{\alpha}(t,y,z)=\frac{1}{\alpha}g\lt(t,\alpha y,\alpha z\rt),  \ dP\times dt\text{-a.s.}$$

(2).  Let $c\in\R$, for any $T'\in[0,T]$ and $(\omega,t,y,z)\in\Omega\times[0,T']\times\R\times\R^d$, define
$$g^{c}(\omega,t,y,z):= g\lt(\omega,t,y-c,z\rt).$$
It is easy to check that $g^{c}$ also satisfies assumptions (A1) and (A3). By the definition and the uniqueness of transposition solution (Definition \ref{de1}, Theorem \ref{th1}), we have that, for any $\xi\in L^2_{\F_{T'}}(\CW,\R)$,
 $$\lt(Y_{t}\lt(g^c,T', \xi+c\rt),Z_{t}\lt(g^c,T', \xi+c\rt)\rt)_{t\in[0,T']}=( Y_{t}\lt(g,T',\xi\rt)+c
 , Z_{t}\lt(g,T',  \xi\rt))_{t\in[0,T']} $$
 in $ L^2_{\bF}(\CW,D([0,T'];\R))\times L^2_{\bF}(\CW,L^2(0,T';\R^{d})).$

Combining this equality with (\ref{traneq}), we obtain for any $t\in[0,T']$, $\xi\in L^2_{\F_{T'}}(\CW,\R)$,
$$Y_{t}\lt(g^c,T', \xi\rt)= Y_{t}\lt(g,T',\xi\rt)\quad P\text{-a.s.}$$
Then by the  converse comparison theorem Theorem \ref{conversethm}, we have for any $c\in\R$,
$$\text{for all } (y,z)\in\R\times\R^d\ g(t,y,z)=g^{c}(t,y,z)=g\lt(t, y-c, z\rt),  \ dP\times dt\text{-a.s.}$$
That is,  $g$ is independent of $y$.

(3). For any fixed $(y_1,z_1),(y_2,z_2)\in\R\times\R^d $, and $t\in S_{y_1+y_2,z_1+z_2}(g)\cap S_{y_1,z_1}(g)\cap S_{y_2,z_2}(g)$, we have
\begin{align*}
g(t,y_1+y_2,z_1+z_2)&=\lim\limits_{n\to\infty}n\lt\{Y_t\lt(g,t+1/n,y_1+y_2+(z_1+z_2)(W_{t+1/n}-W_t)\rt)-y_1-y_2\rt\}\quad\text{ in }L^2;\\
g(t,y_1,z_1)&=\lim\limits_{n\to\infty}n\lt\{Y_t\lt(g,t+1/n,y_1+z_1(W_{t+1/n}-W_t)\rt)-y_1\rt\}\quad\text{ in }L^2;\\
g(t,y_2,z_2)&=\lim\limits_{n\to\infty}n\lt\{Y_t\lt(g,t+1/n,y_2+z_2(W_{t+1/n}-W_t)\rt)-y_2\rt\}\quad\text{ in }L^2.
\end{align*}
Thanks to Corollary \ref{rep-cor} and (\ref{subeq}), we can deduce that $g$ is sub-additive in $(y,z)$.

(4). The proof of (4) is similar to (3), so we omit it.
\end{proof}

\section{Extended $g$-expectation and Conditional $g$-expectation with General Filtration}
To well define $g$-expectation on $L^2_{\F_T}(\CW,\R)$, $g$ is required to satisfy assumptions (A1) and (A5) as in Peng \cite{Peng1997-2}.
\begin{definition}\label{gexpectation}
For any given function $g$ satisfied assumptions (A1) and (A5), and any $Y^T\in L^2_{\F_T}(\CW,\R)$,
let $(Y_{\cdot},Z_{\cdot})\in L^2_{\bF}(\CW,D([0,T];\R))\times L^2_{\bF}(\CW,L^2(0,T;\R^d))$ be the unique transposition solution of BSDE (\ref{bsde}).
 The g-expectation $\HE_g[\cdot] $ of $Y^T$ is defined by $$\HE_g[Y^T]=Y_0.$$
 \end{definition}
\begin{remark}
If $\F_0=\F_0'$, where $\F_0'$ is the $\sigma$-algebra generated by all the $P$-null sets in $\F$,  then $g$-expectation of $Y^T$ is a constant $P$-a.s. If $\F_0$ is larger than $\F_0'$, the $g$-expectation of $Y^T$ may be a random variable. So, in the sequel, $\HE_g[\xi^1]=\HE_g[\xi^2]$ always means that  $\HE_g[\xi^1]=\HE_g[\xi^2]\ P$-a.s. If we want the $g$-expectation of  $Y^T\in L^2_{\F_T}(\CW,\R)$ to be a constant ($P$-a.s.) as the
the classical expectation or nonlinear expectations be,  we can restrict $\F_0$ to be $\F'_0$.

\end{remark}
 \begin{prop}\label{prop1}
Suppose that $g$ satisfies assumptions (A1) and (A5). For any $Y^T\in L^2_{\F_T}(\CW,\R)$, let $(Y_{\cdot},Z_{\cdot})\in L^2_{\bF}(\CW,D([0,T];\R))\times L^2_{\bF}(\CW,L^2(0,T;\R^d))$ be the unique transposition solution of BSDE (\ref{bsde}). Then for any $ t\in[0,T]$
$$\HE_g[Y_t]=\HE_g[Y^T].$$
\end{prop}
\begin{proof}[{\bf Proof}]
We only need to prove that for any $t\in[0,T]$, $\lt(Y_sI_{\{s\le t\}}+Y_{t}I_{\{s> t\}},Z_sI_{\{s\leq t\}}\rt)_{s\in[0,T]}$ is the transposition solution of BSDE (\ref{bsde}) with terminal condition $Y_{t}$. In other words, we need to prove that for any
$s\in [0,T]$, $u_{\cdot}\in L^2_{\bF}(\CW,L^1(s,T;\R))$, $v_{\cdot}\in L^2_{\bF}(\CW,L^2(s,T;\R^{ d})) $ and $\eta\in L^2_{\F_s}(\CW,\R)$, the following identity
holds
\begin{align}\label{eq6}
&E\lt[Y_{t} X_T+\int_s^TX_rg\lt(r,Y_rI_{\{r\leq t\}}+Y_tI_{\{r>t\}},Z_rI_{\{r\leq t\}}\rt)dr\rt]\nonumber\\
=&E\lt[(Y_sI_{\{s\le t\}}+Y_{t}I_{\{s> t\}})\eta+\int_s^Tu_r\lt(Y_rI_{\{r\leq t\}}+Y_{t}I_{\{r>t\}}\right)dr+\int_s^Tv_rZ_rI_{\{r\leq t\}}dr\rt], \quad 0\leq  s\leq T,
\end{align}
where $X_{\cdot}\in L^2_{\bF}(\CW,C([s,T];\R))$  is the unique strong solution of SDE (\ref{sde}) starting from $s$.

If $s> t$, because of the assumption (A5), we have
\begin{align*}
&E\lt[Y_{t} X_T+\int_s^TX_rg\lt(r,Y_rI_{\{r\leq t\}}+Y_tI_{\{r>t\}},Z_rI_{\{r\leq t\}}\rt)dr\rt]
=E\lt[Y_{t} X_T\rt]\\
=&E\lt[Y_{t}E\lt[\lt(\eta+\int_s^Tu_r dr+\int_s^Tv_r dW_r\rt)\big|\mathcal{F}_{t}\rt]\rt]=E\lt[Y_{t}\eta+\int_s^TY_{t}u_r dr\rt]\\
=&E\lt[(Y_sI_{\{s\le t\}}+Y_{t}I_{\{s> t\}})\eta+\int_s^Tu_r\lt(Y_rI_{\{r\leq t\}}+Y_{t}I_{\{r>t\}}\right)dr+\int_s^Tv_rZ_rI_{\{r\leq t\}}dr\rt].
\end{align*}
If $s\leq t$, due to the assumption (A5), we have
\begin{align*}
&E\lt[Y_{t} X_T+\int_s^TX_rg\lt(r,Y_rI_{\{r\leq t\}}+Y_tI_{\{r>t\}},Z_rI_{\{r\leq t\}}\rt)dr\rt]\\
=&E\lt[Y_{t} X_T+\int_s^{t}X_rg(r,Y_r,Z_r)dr\rt]\\
=&E\lt[Y_t\lt(X_{t}+\int_{t}^T u_r dr+\int_{t}^T v_r dW_r\rt)+\int_s^{t}X_rg(r,Y_r,Z_r)dr\rt].
\end{align*}
It follows from Proposition \ref{prop3} that
\begin{align*}
&E\lt[Y_{t} \lt(X_{t}+\int_{t}^T u_r dr+\int_{t}^T v_r dW_r\rt)+\int_s^{t}X_rg(r,Y_r,Z_r)dr\rt]\\
=&E\lt[Y_{t} X_{t}+\int_s^{t}X_rg(r,Y_r,Z_r)dr+Y_{t}\int_{t}^T u_r dr\rt]\\
=&E\lt[Y_s\eta+\int_s^{t}u_rY_rdr+Y_{t}\int_{t}^Tu_rdr+\int_s^{t}v_rZ_rdr\rt]\\
=&E\lt[(Y_sI_{\{s\le t\}}+Y_{t}I_{\{s> t\}})\eta+\int_s^Tu_r\lt(Y_rI_{\{r\leq t\}}+Y_{t}I_{\{r>t\}}\right)dr+\int_s^Tv_rZ_rI_{\{r\leq t\}}dr\rt]
\end{align*}
which shows the identity (\ref{eq6}) hold.
\end{proof}

Further more, we could introduce the conditional $g$-expectation of  $Y^T$ with respect to $\mathcal{F}_t$, $t\in [0,T]$. By analogy with the notion of the classical expectation and Peng's g-expectation (see \cite{Peng1997-2}), we are looking for a random variable $\zeta$ satisfying (\ref{eq1})
\begin{equation}\label{eq1}
\ \left\{\begin{aligned}
(i)&\ \zeta\text{ is }\F_t\text{-measurable and  }\zeta\in L_{\F_t}^2(\CW,\R);\\
(ii)&\ \HE_g[I_AY^T]=\HE_g[I_A\zeta],\ \text{ for all }A\in\F_t.\\
\end{aligned}\right.
\end{equation}
Actually, we have
\begin{thm}\label{condiexp}
Suppose that $g$ satisfies assumptions (A1) and (A5). For any $t\in[0,T]$ and each   $Y^T\in L^2_{\F_T}(\CW,\R)$, there exists a $P$-a.s. unique random variable $\zeta$ in $ L^2_{\F_t}(\CW,\R)$ satisfies  (\ref{eq1}). Furthermore, this $\zeta$ coincides with $Y_t$, the transposition solution of BSDE (\ref{bsde}) at time $t$.
\end{thm}
\begin{definition}\label{de2}
Suppose that $g$ satisfies assumptions (A1) and (A5). For any $t\in[0,T]$, we define the random variable $\zeta$ satisfying (\ref{eq1}) as the conditional g-expectation  of $Y^T$ under $\F_t$, denoted by $\HE_g[Y^T|\F_t]$.
\end{definition}
\begin{remark}
When $\bF$ is the natural filtration generated by the Brownian motion $\{W_t\}_{t\in[0,T]}$ , the  definition of g-expectation and conditional g-expectation here are coincident with the definition introduced by Peng \cite{Peng1997-2} since the transposition solution of BSDE coincides with the usual strong solution in this case (\cite{LuZhang}).
\end{remark}

\begin{proof}[\textup{\bf Proof of Theorem \ref{condiexp}}]
  \textit{Uniqueness:} If both $\zeta_1,\zeta_2$   satisfy  (\ref{eq1}), then we have for all $A\in\F_t$
$$\HE_g[I_A\zeta_1]=\HE_g[I_A\zeta_2].$$
In particular,  setting $A=\{\zeta_1\geq\zeta_2\} $ or $A=\{\zeta_1\leq\zeta_2\}$, we get
$$
\HE_g[I_{\{\zeta_1\geq\zeta_2\}}\zeta_1]=\HE_g[I_{\{\zeta_1\geq\zeta_2\}}\zeta_2],\quad
\HE_g[I_{\{\zeta_1\leq\zeta_2\}}\zeta_1]=\HE_g[I_{\{\zeta_1\leq\zeta_2\}}\zeta_2].
$$
However
$$
I_{\{\zeta_1\geq\zeta_2\}}\zeta_1\geq I_{\{\zeta_1\geq\zeta_2\}}\zeta_2,\quad
I_{\{\zeta_1\leq\zeta_2\}}\zeta_1\leq I_{\{\zeta_1\leq\zeta_2\}}\zeta_2,
$$
It follows from the comparison theorem (Theorem \ref{comparison}) that
$$
I_{\{\zeta_1\geq\zeta_2\}}\zeta_1=I_{\{\zeta_1\geq\zeta_2\}}\zeta_2,\ P\text{-a.s.},\quad
I_{\{\zeta_1\leq\zeta_2\}}\zeta_1=I_{\{\zeta_1\leq\zeta_2\}}\zeta_2,\ P\text{-a.s.}
$$
Thus $\zeta_1=\zeta_2,\ P$-a.s.

\textit{Existence:} Let  $(Y_s,Z_s)_{s\in[0,T]}$ be the transposition solution of BSDE (\ref{bsde}) with $Y_T=Y^T$, that is
\begin{equation}\label{eq2}
E\lt[Y^T X_T+\int_s^TX_rg(r,Y_r,Z_r)dr\rt]=E\lt[Y_s\eta+\int_s^Tu_rY_rdr+\int_s^Tv_rZ_rdr\rt], \text{ for all }  s\in[0,T],
\end{equation}
and let
 $(Y'_s,Z'_s)_{s\in[0,T]}$ be the transposition solution of BSDE (\ref{bsde}) with $Y'_T=I_AY^T$, that is,
\begin{equation}\label{eq3}
E\lt[I_AY^T X_T+\int_s^TX_rg(r,Y'_r,Z'_r)dr\rt]=E\lt[Y'_s\eta+\int_s^Tu_rY'_rdr+\int_s^Tv_rZ'_rdr\rt], \text{ for all } s\in[0,T],
\end{equation}
where $A\in \mathcal{F}_t$, $X_{\cdot} \in L^2_{\bF}(\CW,C([s,T];\R))$ is the unique strong solution of the following SDE with any fixed $ u_{\cdot}\in L^2_{\bF}(\CW,L^1(s,T;\R)),v_{\cdot}\in L^2_{\bF}(\CW,L^2(s,T;\R^d)) $ and $\eta\in L^2_{\F_s}(\CW,\R)$,
\begin{equation*}
\left\{\begin{aligned}
&dX_r=u_rdr+v_rdW_r,\quad r\in[s,T]\\
&X_s=\eta\\
\end{aligned}\right..
\end{equation*}
For $s\in[t,T]$, we have
\begin{equation*}
\left\{\begin{aligned}
&dI_AX_r=I_Au_rdr+I_Av_rdW_r,\quad r\in[s,T]\\
&I_AX_s=I_A\eta\\
\end{aligned}\right..
\end{equation*}
It follows from (\ref{eq2}) that\small
\begin{equation*}
E\lt[Y^T (X_TI_A)+\int_s^T(X_rI_A)g(r,Y_r,Z_r)dr\rt]=E\lt[Y_s( I_A\eta)+\int_s^T(u_rI_A)Y_rdr+\int_s^T(v_rI_A)Z_rdr\rt], \text{ for all } s\in[t,T].
\end{equation*}\normalsize
Due to assumption (A5), we obtain\small
\begin{equation}\label{eq4}
E\lt[(I_AY^T) X_T+\int_s^TX_rg(r,I_AY_r,I_AZ_r)dr\rt]=E\lt[(I_AY_s)\eta+\int_s^Tu_r(I_AY_r)dr+\int_s^Tv_r(I_AZ_r)dr\rt], \text{ for all } s\in[t,T].
\end{equation}\normalsize
Comparing (\ref{eq3}) and (\ref{eq4}), we have $(Y'_{\cdot},Z'_{\cdot})=(I_AY_{\cdot},I_AZ_{\cdot})$ in $L^2_{\bF}(\CW,D([t,T];\R))\times L^2_{\bF}(\CW,L^2(t,T;\R^d))$
Thus, by Proposition \ref{prop1}, we have
$\HE_g[I_AY_t]=\HE_g[Y'_t]=\HE_g[Y'_T]=\HE_g[I_AY^T]$.
The proof is completed.
\end{proof}

\begin{thm}\label{prop}
For any given function $g$ satisfied assumptions (A1) and (A5),  we have the following properties:
\begin{enumerate}[(1).]
  \item  \textup{\bf(Monotonicity)} For any $\xi_1,\xi_2\in L^2_{\F_T}(\CW,\R)$, and if $\xi_1\geq\xi_2$ P-a.s., then for any $t\in[0,T]$ we have $\HE_g[\xi_1|\F_t]\geq\HE_g[\xi_2|\F_t]$ P-a.s. Moreover, $\HE_g[\xi_1|\F_t]=\HE_g[\xi_2|\F_t]$ P-a.s. for some $t\in[0,T]$, if and only if $\xi_1=\xi_2$ P-a.s.\\
  For any $t\in[0,T]$, $\xi_1,\xi_2\in L^2_{\F_t}(\CW,\R)$, $\xi_1\ge\xi_2$ $P$-a.s. if and only if for all $A\in\mathcal{F}_t$, $\HE_g[\xi_1I_A]\geq\HE_g[\xi_2I_A]$.
  \item \textup{\bf(``Zero-one'' law)} For all $\xi\in L^2_{\F_T}(\CW,\R)$, $B\in\mathcal{F}_t$, we have $\HE_g[I_B\xi|\F_t]=I_B\HE_g[\xi|\F_t]$.
  \item For any $\xi_1,\xi_2\in L^2_{\F_T}(\CW,\R)$ and $t\in[0,T]$, there exists a conxtant $C>0$ such that
      $$E\lt[\lt|\HE_g[\xi_1|\F_t]-\HE_g[\xi_2|\mathcal{F}_t]\rt|^2\rt]\leq CE[|\xi_1-\xi_2|^2]$$
  \item For each $t\in[0,T]$, if $\xi\in L^2_{\F_t}(\CW,\R)$, then $\HE_g[\xi|\F_t]=\xi$.
  \item  \textup{\bf(Constant preserving)} For any $t\in[0,T]$ and $c\in\R$, $\HE_g[c|\F_t]=c$.
  \item \textup{\bf(Time consistency)} For all $\xi\in L^2_{\F_T}(\CW,\R)$ and $t_1,t_2\in[0,T]$,  $\HE_g[\HE_g[\xi|\F_{t_2}]|\F_{t_1}]=\HE_g[\xi|\F_{t_1\wedge t_2}]$.

  \end{enumerate}
  \end{thm}
\begin{proof}[\textup{\bf Proof}]
(1). This can be deduced from Theorem \ref{comparison} directly.

(2). For each $A\in\mathcal{F}_t$, we have
$$\HE_g[I_AI_B\xi]=\HE_g[I_{A\cap B}\xi]=\HE_g[I_{A\cap B}\HE_g[\xi|\mathcal{F}_t]]=\HE_g[I_A(I_B\HE_g[\xi|\mathcal{F}_t])].$$
It follows from the definition of the conditional $g$-expectation (Definition \ref{de2}) that $\HE_g[I_B\xi|\F_t]=I_B\HE_g[\xi|\F_t]$.

(3). For any $t\in[0,T]$, suppose  $(Y_r^i,Z_r^i)_{r\in[t,T]}$ is the unique one satisfied  (\ref{transp}) with $Y^T=\xi_i$, i=1,2. Then, we have
\begin{align*}
&E\lt[(\xi_1-\xi_2)X_T+\int_t^TX_r\lt(g(r,Y_r^1,Z_r^1)-g(r,Y_r^2,Z_r^2)\rt)dr\rt]\\
=&E\lt[(Y_t^1-Y_t^2)\eta+\int_t^Tu_r(Y_r^1-Y_r^2)dr+\int_t^Tv_r(Z_r^1-Z_r^2)dr\rt].
\end{align*}
\normalsize
For any $(\omega,t,y,z)\in\CW\times[0,T]\times\R\times\R^d $, define $$\bar{g}(\omega,t,y,z)=g(\omega,t,y+Y_t^2(\omega),z+Z_t^2(\omega))-g(\omega,t,Y_t^2(\omega),Z_t^2(\omega)).$$ It can be easily checked that $\HE_{\bar{g}}[\xi_1-\xi_2|\F_t]=Y_t^1-Y_t^2=\HE_g[\xi_1|\F_t]-\HE_g[\xi_2|\F_t]$. By theorem \ref{th2}, we have
    $$E\lt[\lt|\HE_g[\xi_1|\F_t]-\HE_g[\xi_2|\F_t]\rt|^2\rt]=E[|\HE_{\bar{g}}[\xi_1-\xi_2|\F_t]|^2]\leq CE[|\xi_1-\xi_2|^2].$$

(4). We only need to prove $(Y_r,Z_r)_{r\in[t,T]}=(\xi,0)$ satisfying the (\ref{transp}) with $Y^T=\xi \in L^2_{\F_t}(\CW,\R)$. Due to the assumption (A5), the left hand of  (\ref{transp}) is
$$E\lt[\xi X_T+\int_t^TX_rg(r,\xi,0)dr\rt]=E[\xi X_T].$$
Meanwhile, the right hand of  (\ref{transp}) is
\begin{align*}
E\lt[\xi \eta+\int_t^T\xi u_rdr+\int_t^T0\cdot v_r dr\rt]&=E\lt[\xi(\eta+\int_t^Tu_rdr)\rt]\\
&=E\lt[\xi E\lt[\lt. \eta+\int_t^Tu_rdr\rt|\F_t\rt]\rt]\\
&=E\lt[\xi E\lt[\lt.\eta+\int_t^Tu_rdr+\int_t^Tv_rdW_r\rt|\F_t\rt]\rt]\\
&=E[\xi X_T].
\end{align*}
Therefore, $\HE_g[\xi|\F_t]=\xi$.

(5). It is a special case of (4).

(6).  The method to prove (6) is similar to the method used in the proof of Proposition \ref{prop1}, so we omit it.
\end{proof}

\begin{thm} Let $g_1,g_2$ be two functions satisfied assumptions (A1) and (A5). \begin{enumerate}[(1).]
\item  If for all $(y,z)\in\R\times\R^d, $  $g_1(t,y,z)\geq g_2(t,y,z)$ $dP\times dt$-a.s., then for any $\xi\in L^2_{\mathcal{F}_T}(\Omega,\mathbb{R})$, $t\in [0,T]$, we have $\HE_{g_1}[\xi|\mathcal{F}_t]\geq \HE_{g_2}[\xi|\mathcal{F}_t]$ $P$-a.s.
 \item If for all $\xi\in L^2_{\mathcal{F}_T}(\Omega,\mathbb{R})$, $\HE_{g_1}[\xi]=\HE_{g_2}[\xi]$, then for all $t\in[0,T]$, $\HE_{g_1}[\xi|\mathcal{F}_t]=\HE_{g_2}[\xi|\mathcal{F}_t]$ $P$-a.s.
\item If for all $\xi\in L^2_{\mathcal{F}_T}(\Omega,\mathbb{R}), t\in[0,T]$, $\HE_{g_1}[\xi|\mathcal{F}_t]\geq\HE_{g_2}[\xi|\mathcal{F}_t] \ P\text{-a.s.},$
then we have
$$\text{for all } (y,z)\in\R\times\R^d ,\quad g_1(t,y,z)\geq g_2(t,y,z),  \quad dP\times dt\text{-a.s.}$$
In particular, for any  $\xi\in L^2_{\mathcal{F}_T}(\Omega,\mathbb{R})$, $\HE_{g_1}[\xi]=\HE_{g_2}[\xi]$ if and only if
$$\text{for all } (y,z)\in\R\times\R^d ,\quad g_1(t,y,z)= g_2(t,y,z),  \quad dP\times dt\text{-a.s.}$$
\end{enumerate}
\end{thm}
\begin{proof}[\textup{\bf Proof}]
(1). The proof follows directly from Theorem \ref{comparison}.

(2). By the definition of  conditional g-expectation, for any fixed $\xi\in L^2_{\F_T}(\CW,\R)$, any $t\in[0,T]$ and $A\in\F_t$, we have
$$\HE_{g_1}[I_A\HE_{g_1}[\xi|\F_t]]=\HE_{g_1}[I_A\xi]=\HE_{g_2}[I_A\xi]=\HE_{g_2}[I_A\HE_{g_2}[\xi|\F_t]]=\HE_{g_1}[I_A\HE_{g_2}[\xi|\F_t]].$$
From the uniqueness in Theorem \ref{condiexp}, we have
$$\HE_{g_1}[\xi|\F_t]=\HE_{g_2}[\xi|\F_t],\quad P\text{-a.s.}$$

(3). The proof follows directly from Theorem \ref{conversethm}
\end{proof}

For studying dynamic risk measures, Rosazza Gianin \cite{Gianin} and Jiang \cite{Jiang2008} studied the
 positive homogeneity, translation invariance, sub-additivity and convexity of classical g-expectations and conditional g-expectations. Thanks to Theorem \ref{th1} and \ref{converprop}, we also establish the necessary
and sufficient conditions for positive homogeneity, translation invariance, subadditivity and convexity
of the extended g-expectations, respectively, which extend the Theorem 3.1-3.4 of \cite{Jiang2008} to the  general filtration probability spaces.

\begin{coro}[{\bf Positive homogeneity}]\label{Positive}
Let $g$ be any given function satisfied assumptions (A1) and (A5), then the following
statements are equivalent,
  \begin{enumerate}[(1).]
   \item $ \HE_g[\cdot] $ is positive homogeneous;
   \item $ \HE_g[\cdot|\F_t] $ is positive homogeneous for any  $t\in[0,T]$, that is, for any $\xi\in L^2_{\F_T}(\CW,\R)$ and $\alpha\geq0$,
      $$\HE_g[\alpha \xi|\F_t]=\alpha\HE_g[\xi|\F_t]\ P\text{-a.s.};$$
   \item   $g$ is positive homogeneous in $(y,z)$.
    \end{enumerate}
\end{coro}

\begin{coro}[{\bf Translation invariance}]\label{Translation}
Let $g$ be any given function satisfied assumptions (A1) and (A5), then the following
statements are equivalent,
\begin{enumerate}[(1).]
\item For any  $ \xi\in L^2_{\F_T}(\CW,\R),\ c\in\R$, we have $\HE_g[\xi+c]=\HE_g[\xi]+c$;
\item For any  $t\in[0,T],\xi\in L^2_{\F_T}(\CW,\R),\beta\in L^2_{\F_t}(\CW,\R)$, we have
      $\HE_g[\xi+\beta|\F_t]=\HE_g[\xi|\F_t]+\beta\ P\text{-a.s.};$
\item  $g$ is independent of $y$.
\end{enumerate}
\end{coro}
%
%
%
%

\begin{coro}[{\bf Sub-additivity}]\label{Subadditivity}
Let $g$ be any given function satisfied assumptions (A1) and (A5), then the following
statements are equivalent,
\begin{enumerate}[(1).]
\item $ \HE_g[\cdot] $ is sub-additive;
\item $ \HE_g[\cdot|\F_t] $ is sub-additive for any  $t\in[0,T]$, that is, for any $\xi_1,\xi_2\in L^2_{\F_T}(\CW,\R)$,
      $$\HE_g[\xi_1+\xi_2|\F_t]\leq\HE_g[\xi_1|\F_t]+\HE_g[\xi_2|\F_t]\ P\text{-a.s.};$$
\item  $g$ is independent of $y$ and $g$ is sub-additive with respect to $z$.
\end{enumerate}
\end{coro}
\begin{proof}[\textup{\bf Proof}] Thanks to Theorem \ref{th1} (3), we only need to prove
  $(1)\Rightarrow(3).$
 We first prove that (1) implies the translation
invariance property.  Suppose that (1) holds for $\HE_g[\cdot]$, then for any $\xi\in L^2_{\F_T}(\CW,\R)$ and $c\in\R$, we have
$$\HE_g[\xi+c]\leq\HE_g[\xi]+\HE_g[c]=\HE_g[\xi]+c,$$
and
$$\HE_g[\xi]=\HE_g[\xi+c-c]\leq\HE_g[\xi+c]+\HE_g[-c]=\HE_g[\xi+c]-c,$$
that is $\HE_g[\xi+c]=\HE_g[\xi]+c$. It follows from Theorem \ref{converprop} (2) that $g$ is independent of $y$.

Now we prove that (1) implies (2). For any $\xi_1,\xi_2\in L^2_{\F_T}(\CW,\R)$ and $t\in[0,T]$, set
 $$A_t:=\lt\{ \HE_g[\xi_1+\xi_2|\F_t]>\HE_g[\xi_1|\F_t]+\HE_g[\xi_2|\F_t]\rt\},$$
thus $A_t \in \F_t$. If $P(A_t)>0$, then we have
$$I_{A_t}\HE_g[\xi_1+\xi_2|\F_t]-I_{A_t}\lt(\HE_g[\xi_1|\F_t]+\HE_g[\xi_2|\F_t]\rt)\geq0,$$
and
$$P\Big(I_{A_t}\HE_g[\xi_1+\xi_2|\F_t]-I_{A_t}\lt(\HE_g[\xi_1|\F_t]+\HE_g[\xi_2|\F_t]\rt)>0\Big)>0.$$
Since $g$ is independent of $y$, it follows from Theorem \ref{prop} (2)(6) and Theorem \ref{th1} (2) that
\begin{align*}
&\HE_g\lt[ I_{A_t}(\xi_1-\HE_g[\xi_1|\F_t])+I_{A_t}(\xi_2-\HE_g[\xi_2|\F_t])\rt]\\
=& \HE_g\big[\HE_g\lt[ I_{A_t}(\xi_1-\HE_g[\xi_1|\F_t])+I_{A_t}(\xi_2-\HE_g[\xi_2|\F_t])\big|\mathcal{F}_t\rt]\big]\\
=& \HE_g\big[\HE_g\lt[ I_{A_t}(\xi_1+\xi_2)\big|\mathcal{F}_t]-I_{A_t}(\HE_g[\xi_1|\F_t\rt]+\HE_g[\xi_2|\F_t])\big]\\
=&\HE_g\lt[ I_{A_t}\HE_g[\xi_1+\xi_2|\F_t]-I_{A_t}(\HE_g[\xi_1|\F_t]+\HE_g[\xi_2|\F_t])\rt]>0.
\end{align*}
Similarly, we have
\begin{align*}
\HE_g\lt[ I_{A_t}(\xi_1-\HE_g[\xi_1|\F_t])\rt]=\HE_g\lt[ \HE_g\lt[I_{A_t}(\xi_1-\HE_g[\xi_1|\F_t])|\F_t\rt]\rt]=\HE_g\lt[ \HE_g\lt[I_{A_t}\xi_1|\F_t\rt]-I_{A_t}\HE_g[\xi_1|\F_t]\rt]=0,\\
\HE_g\lt[ I_{A_t}(\xi_2-\HE_g[\xi_2|\F_t])\rt]=\HE_g\lt[ \HE_g\lt[I_{A_t}(\xi_2-\HE_g[\xi_2|\F_t])|\F_t\rt]\rt]=\HE_g\lt[ \HE_g\lt[I_{A_t}\xi_2|\F_t\rt]-I_{A_t}\HE_g[\xi_2|\F_t]\rt]=0.
\end{align*}
Thus
$$\HE_g\lt[ I_{A_t}(\xi_1-\HE_g[\xi_1|\F_t])+I_{A_t}(\xi_2-\HE_g[\xi_2|\F_t])\rt]>0=\HE_g\lt[ I_{A_t}(\xi_1-\HE_g[\xi_1|\F_t])\rt]+\HE_g\lt[ I_{A_t}(\xi_2-\HE_g[\xi_2|\F_t])\rt],$$
which is a contradiction to (1). Hence, $P(A_t)=0$ and (2) follows. Combining with Theorem \ref{converprop} (3), we can get (3).

\end{proof}

\begin{coro}[{\bf Convexity}]\label{convex}
Let $g$ be any given function satisfied assumptions (A1) and (A5), then the following
statements are equivalent,
\begin{enumerate}[(1).]
\item $ \HE_g[\cdot] $ is convex;
\item $ \HE_g[\cdot|\F_t] $ is convex for any  $t\in[0,T]$, that is, for any  $\xi_1,\xi_2\in L^2_{\F_T}(\CW,\R)$ and $\alpha\in[0,1]$,
     $$\HE_g[\alpha\xi_1+(1-\alpha)\xi_2|\F_t]\leq\alpha\HE_g[\xi_1|\F_t]+(1-\alpha)\HE_g[\xi_2|\F_t]\ P\text{-a.s.};$$
\item   $g$ is independent of $y$ and $g$ is convex with respect to $z$.
\end{enumerate}
\end{coro}
\begin{proof}[\textup{\bf Proof}] Due to Theorem \ref{th1} (4), we only need to prove $(1)\Rightarrow(3).$
 We first prove that (1) implies the translation
invariance property.  Suppose that (1) holds for $\HE_g[\cdot]$, then for any $\xi\in L^2_{\F_T}(\CW,\R)$ , $c\in\R$ and $n\geq1$, we have
$$\HE_g\lt[\lt(1-\frac{1}{n}\rt)\xi+c\rt]=\HE_g\lt[\lt(1-\frac{1}{n}\rt)\xi+\frac{1}{n}(nc)\rt]\leq\lt(1-\frac{1}{n}\rt)\HE_g[\xi]+c.$$
Then, it follows  Theorem \ref{prop} (3) that
$$\HE_g[\xi+c]=\lim\limits_{n\to\infty}\HE_g\lt[\lt(1-\frac{1}{n}\rt)\xi+c\rt]\leq\lim\limits_{n\to\infty}\lt(1-\frac{1}{n}\rt)\HE_g[\xi]+c=\HE_g[\xi]+c.$$
Meanwhile,
$$\HE_g[\xi]=\HE_g[\xi+c-c]\leq\HE_g[\xi+c]-c,$$
thus, $\HE_g[\xi+c]=\HE_g[\xi]+c$. Then $g$ is independent of $y$ follows from Theorem \ref{converprop} (2).

Then we can prove $(1)\Rightarrow(2)$ similarly as the proof of Corollary \ref{Subadditivity}. Therefore, we get (3) from (2) and Theorem \ref{converprop} (4).

\end{proof}

\hspace{-0.6cm}\textbf{Acknowledgements}\\The work is supported by the National Key R$\&$D Program of China (Grant No. 2018YFA0703900) and the National Natural Science Foundation of China (Grant No. 11601280) and the Natural Science Foundation of Shandong Province of China
(Grant Nos. ZR2016AQ11 and ZR2016AQ13).

\end{document}